\documentclass[a4paper,12pt]{amsart}
\usepackage{amsmath,amsthm}
\usepackage{amsfonts,amsmath,amsthm, amssymb, mathrsfs}
\usepackage{float}
\usepackage{euscript,verbatim}
\usepackage{graphicx}
\usepackage[usenames]{color}
\usepackage[colorlinks,linkcolor=red,anchorcolor=blue,citecolor=blue]{hyperref}
\usepackage{amsmath}
\usepackage{amsthm}
\usepackage[all]{xy}
\usepackage{graphicx} 
\usepackage{amssymb}

\newtheorem{thm}{Theorem}[section]
\newtheorem{prop}[thm]{Proposition}
\newtheorem{df}[thm]{Definition}
\newtheorem{lem}[thm]{Lemma}

\newtheorem{claim}[thm]{Claim}

\newtheorem{rem}[thm]{Remark}

\newtheorem*{thm1}{Theorem A}

\def\R{\mathbb{R}}

\def\FF{\mathcal{F}}

\def\PP{\mathcal{P}}

\def\UU{\mathcal{U}}

\def\diam{\text{\rm diam}}

\def\Leb{\text{\rm Leb}}

\def\AA{\mathcal{A}}
\def\BB{\mathcal{B}}

\def\CC{\mathcal{C}}
\def\VV{\mathcal{V}}

\makeatletter 
\@addtoreset{equation}{section}
\numberwithin{equation}{section}

\title{Variational principle for   weighted  amenable  topological  pressure}

\author{Jiao Yang, Ercai Chen,  Rui Yang*  and Xiaoyi Yang
}
\address
{1.School of Mathematical Sciences and Institute of Mathematics, Nanjing Normal University, Nanjing 210023, Jiangsu, P.R.China}

\email{jiaoyang6667@126.com;ecchen@njnu.edu.cn;zkyangrui2015@163.com; 1197380824@qq.com}

\date{}

\begin{document}

\renewcommand{\thefootnote}{}
\footnote{*corresponding author}

\begin{abstract}
	This paper aims to  investigate the  thermodynamic formalism of weighted amenable topological pressure for factor maps  of amenable group actions. Following the  approach of Tsukamoto [\emph{Ergodic Theory Dynam. Syst.} \textbf{43}(2023), 1004-1034.], we introduce the notion of  weighted   amenable   topological pressure for  factor maps  of  amenable group actions, and establish a variational principle for it. As the application of variational principle, we show weighted  amenable measure-theoretic entropy can be determined by weighted  amenable topological pressure. Equilibrium states of  weighted topological pressure are also involved.
	
\end{abstract}
\maketitle

\section{Introduction}
For  measure-preserving systems,  measure-theoretic entropy of
invariant measures    was introduced by Kolmogorov \cite{kol58} and Sinal \cite{s59}. The topological entropy was first introduced by Adler, Konheim and  McAndrew \cite{akm65} for topological dynamical system to capture the topological complexity of systems.  The  classical variational principle \cite{goodw69,din70,goodm71,mis75}  provides the bridge between  ergodic theory and  topological dynamics:  
   $$h_{top}(X,T)=\sup_{\mu \in M(X,T)}h_{\mu}(T),$$
   where  $h_{top}(X,T),  h_{\mu}(T)$ and $M(X,T)$  respectively  denote the topological entropy  of $X$, the measure-theoretic entropy of $\mu$,  the set of $T$-invariant Borel probability measure on $X$.  Inspired by statistical mechanics, Ruelle \cite{rue73} introduced the notion of topological pressure for expansive topological dynamical systems. Later, Walters extended this notion to general dynamical systems and proved a variational principle for topological pressure, which states  that for  any continuous function $f$ on $X$
   $$P(T,f)=\sup_{\mu \in M(X,T)}\{h_{\mu}(T)+\int fd\mu\},$$ 
   where  $P(T,f)$   denotes the topological pressure of $f$.  It turns out that topological pressure  plays a vital role  in dimension theory and becomes a  fundamental tool  in other fields of dynamical systems. Moreover, topological pressure, the associated variational principle and equilibrium measures are the hardcore  part of thermodynamic formalism.  Since then,  looking for proper variational principles  for topological entropy in terms of  general  group action  has gained abundant attentions.   Especially for (countable) amenable group actions,  entropies in both topological and measure-theoretical  settings were  given in \cite{mp82,ow87}.  The  variational principle for amenable topological entropy was established in \cite{mp82,st80,m85,kl16}. Later, Yan \cite{y15} established conditional variational principle for amenable conditional entropy. The variational principles on compact subsets for amenable Bowen and packing  topological entropies were established by the authors \cite{zc16, hlz20}, \cite{dzz20}, respectively.  The authors \cite{hyz11,ly12,dz15} obtained the local variational principles  for  topological entropy of  open covers under amenable group actions. Besides, a Ledrappier's type
   variational principle  for relative topological tail entropy for factor maps of amenable group actions was established by Zhang and Zhu  \cite{zz23}.

   In 2016,  for $\mathbb{Z}$-actions  Feng and Huang posed the following  question \cite[Question 1.1]{fw1616}:
   
   Question:  How can one define a meaningful term  $P^{(a_1,a_2)}(T,f)$ such that the following variational principle holds?
   $$P^{(a_1,a_2)}(T,f)=\sup_{\mu \in M(X,T)}\{a_1h_{\mu}(T)+a_2h_{\pi_{*}\mu}(S)+\int fd\mu\},$$
   where $\pi$ is the factor map
    between the systems $(X,T)$ and $(Y,S)$, $a_1, a_2$ are two real number with $a_1>0$ and $a_2\geq 0$.  $\pi_{*}\mu$ is the push-forward of $\mu$.
   
   For early exploration toward 
   this  question, 
   the authors \cite{f11, bf12} defined such a  quantity $P^{(a_1,a_2)}(T,f)$ 
     for relative and sub-additive thermodynamic
    formalisms  for factor maps between subshifts  over finite alphabets. Their method highly relies on certain special property of subshifts
   and fails  to  extend to general topological dynamical systems.   Inspired by fractal geometry of self-affine carpets and sponges \cite{bed84,mc84,kp96}, Feng and Huang defined the weighted topological pressure   by means of Carath\'eodory-Pesin structures \cite{b73,p97}  and proved  a variational principle for it \cite[Theorem 1.4]{fw1616}. The random case was considered by Yang et al. \cite{yclz22}.  Later, Zhu \cite{z22} introduced weighted topological entropy for a given open  cover and weighted measure-theoretic entropy of invariant measures for a given  open cover.  The two kinds of different entropies are related by  a local variational principle \cite[Theorem 2.9]{z22}. Moreover,  Zhu  showed the weighted topological entropy(taking  supremum for the weighted topological entropy  over open covers) coincides with the entropy defined by Feng and Huang.  Motivated by the study of mean Hausdorff dimension, a quantity introduced by Lindenstrauss and Tsukamoto  by dynamicalizing the definition of Hausdorff dimension\cite{lt19}, of certain infinite dimensional fractals \cite{t22},  Tsukamoto \cite{t23}   gave a new approach for weighted topological pressure $P^{\omega}(\pi,T,f)$  whose definition is analogous to  the   spanning sets for classical topological pressure. Compared with \cite{fw1616}, Tsukamoto's  approach has more advantages in calculating, see \cite[Example 1.6]{t23}.    Furthermore, Tsukamoto  verified the  equivalence of the two  types of weighted  topological entropy  and established the following variational principle:  for any $0\leq \omega \leq 1$ and $f\in C(X,\mathbb{R})$
   $$P^{\omega}(\pi,T,f)=\sup_{\mu \in M(X,T)}\{\omega h_{\mu}(T)+(1-\omega)h_{\pi_{*}\mu}(S)+\omega \int fd\mu\}.$$
  The other topics involved weighted topological entropy can be found in \cite{zczz18,wh19,sxz20}.
   
   An natural question is whether we can establish variational principle for weighted   amenable  topological pressure. To this end, we define so-called weighted amenable topological pressure by following Tsukamoto's  approach,  and a variational principle is established for it.   Analogous to the  standard \emph{amplification trick} used in the classical variational principle for topological entropy \cite{mis75}, for $\mathbb{Z}$-actions Tsukamoto  vanished the ``small quantity" by  power property of weighted topological pressure. However, it may fail to obtain such power property for  amenable group action. Different  from  the proof of \cite[Proposition 4.1]{t23},  we   overcome this difficulty by employing the local entropy theory   of amenable group developed in \cite{hyz11,ly12,dz15}. We came up with  this method inspired the work of Gutman and \'Spiewak \cite{gs20} who have used local  variational principle for metric mean dimension.   For the upper bound, we borrow the  tool called the principle extension of amenable group \cite{d11,dh13,huc21} developed by Downarowicz and Huczek and reduce the proof to zero-dimensional case, then following some arguments\cite{mis75,t23} we produce a desired  invariant measure. Now,  we state the main result of this paper as follows:
     
     \begin{thm}\label{thm 1.1}
     	Let  $G$ be a countable, discrete amenable group. Suppose that
      $\pi:X\rightarrow Y$ is a factor map between two $G$-systems and $f\in C(X,\mathbb{R})$. Then for each $0\leq\omega\leq1$,
     	$$P^{\omega}(\pi,f,G)=\sup_{\mu\in M(X,G)} \left\{\omega h_{\mu}(G,X)+(1-\omega)h_{\pi_{*}\mu}(G,Y)+\omega\int_X fd\mu\right\},$$
     	 where $P^{\omega}(\pi,f,G)$ denotes  the  $\omega$-weighted amenable topological pressure of $f$.
     \end{thm}
   
 The rest of this paper is organized as follows. In section 2,  we  give the precise definition of weighted   amenable  topological pressure and prepare some ingredients for the proof of Theorem 1.1.   In  section 3, we prove  Theorem  1.1.  

\section{Preliminary}

\subsection{Some backgrounds of amenable group}

 Let $G$ be a countable discrete infinite  group and $\FF(G)$   denote  the set of all finite non-empty subsets of $G$.  A group $G$ is said to an \emph{amenable group} if there exists $\{F_n\}_{n\geq 1} \subset \FF(G)$  such that for any $g\in G$,
$$\lim_{n\to\infty}\frac{|gF_n\triangle F_n|}{|F_n|}=0,$$
where  $|E|$  is counting measure of $E$ and $gF_n\triangle F_n=(gF_n\backslash F_n)\cup (F_n\backslash gF_n)$.
Such sequence is called F\o lner sequence. For  instances,  all finite groups and Abelian
groups, all finitely generated groups of subexponential growth are amenable group.  A example  of nonamenable groups is the free group of rank 2.

 Let  $K,\Omega \in \FF(G)$.  The $K$-boundary of $\Omega$ is $$\partial_K(\Omega)=\{g\in G:Kg\cap \Omega\neq\emptyset~{\rm and}~Kg\cap (G\setminus \Omega)\neq\emptyset\}.$$ It is  readily to show for any F$\phi$lner sequence of $G$,  one has
$$\lim_{n\to\infty}\frac{|\partial_K(F_n)|}{|F_n|}=0$$
for any $K\in\FF(G)$.

A set function $h:\FF(G)\rightarrow \R$  is
\begin{enumerate}
	\item  monotone if $h(E)\le h(F)$ for  $\forall E,F\in\FF(G)$ with $E\subset F$;\\
    \item  $G$-invariant if $h(Eg)= h(E)$  for $\forall g\in G$, $\forall E\in\FF(G)$; \\
	\item sub-additive if $h(E\cup F)\le h(E)+h(F)$   for  $E,F\in\FF(G)$ with $E\cap F=\emptyset$.
\end{enumerate}
\begin{flushleft}

\end{flushleft}
 The following well-known \emph{Ornstein-Weiss theorem}  plays a vital role in the definitions of some entropy-like quantities.
\begin{lem}\cite{ow87,gromov,lw00}\label{lem 2.1}
	 Let $G$ be a countable amenable group. Let  $h:\FF(G)\rightarrow \R$  be a monotone $G$-invariant sub-additive set function. Then there exists $\lambda\in [-\infty,\infty)$(only depending on $G$ and $h$) so that
	$$\lim_{n\to\infty}\frac{h(F_n)}{|F_n|}=\lambda$$
	 for all F\o lner sequences $\{F_n\}_{n\ge1}$  of $G$.
\end{lem}
\subsection{Weighted topological pressure}

 By a pair $(X,G)$ we mean a \emph{$G$-system}, where $X$ is a  compact  metric space  and $\Gamma: G\times X \rightarrow X$ is a continuous mapping satisfying: 
 \begin{enumerate}
 \item   $\Gamma(e_G,x)=x$ for every $x\in X$; 
 \item  $\Gamma(g_1,\Gamma(g_2,x)=\Gamma(g_1g_2,x)$ for every $g_1,g_2\in G$ and $x\in X$.
 \end{enumerate}

 We sometimes write $gx:=\Gamma(g,x)$  when $g$ is fixed.
From now on, the group $G$ is  always assumed  to be a  countable discrete amenable group.   Let $(X,G)$ and $(Y,G)$ be two $G$-systems.  $X$ is said to be a \emph{extension} of $Y$(equivalently, $Y$ is a \emph{factor} of $X$) if there exists  a  continuous surjective mapping $\pi:X \rightarrow Y$ such that $g\pi(x)=\pi(gx)$ for  all $g\in G$ and $x\in X$.

Let $C(X,\mathbb{R})$ denote the Banach space consisting of all continuous functions on $X$ endowed with supremum norm. Given  $E\in \FF(G)$, we set $S_{E}f(x)=\sum_{g\in E}f(gx)$. Let $(X,d), (Y,d^{'})$ be two compact metric spaces and  $\pi:(X,G)\rightarrow (Y,G)$ be the factor map between two $G$-systems. Put  $$d_{E}(x_1,x_2)=\max_{g\in E}d(gx_1,gx_2),~~d^{'}_{E}(y_1,y_2)=\max_{g\in E}d^{'}(gy_1,gy_2)$$ for  any $x_1,x_2\in X$ and $y_1,y_2\in Y$.

\begin{df}
Let $\pi:(X,G)\rightarrow (Y,G)$ be a factor map between two $G$-systems and $f\in C(X, \mathbb{R})$. For $\epsilon>0$, $\omega \in [0,1]$, $F \in \FF(G)$, $\emptyset \not=  \Omega \subset X$,   we define
\begin{align*}
&P(\Omega,f,F,\epsilon)\\
=&\inf\left\{\sum_{i=1}^{k}e^{\sup_{U_i}
	S_{F}f}: \begin{array}{l} \text{there exist open subsets}~U_1,\cdots, U_k\text{of }X
	\text{with}\\
	\Omega\subset U_1\cup\cdots\cup U_k \text{ and }\diam(U_i,d_{F})<\epsilon \text{ for}\\
	\text{all} ~ 1\leq i \leq k
\end{array}\right\}.
\end{align*}
and 
\begin{align*}
&P^{\omega}(\pi,f,F,\epsilon)\\
=&\inf\left\{\sum_{i=1}^k
(P(\pi^{-1}(V_i),f,F,\epsilon))^{\omega}: \begin{array}{l}
	Y=V_1\cup\cdots\cup V_k \text{ is  an open }\\\text{  cover with}~\diam(V_i,d_{F}^{'})<\epsilon\\
	\text{ for all }~ 1\leq i \leq k
\end{array}\right\}.
\end{align*}

It's readily to check that  $ F\in\FF(G)\mapsto \log P^{\omega}(\pi,f,F,\epsilon)$  is  sub-additive. By Ornstein-Weiss theorem (see Lemma \ref{lem 2.1}),   the limit
$$\lim_{n\to\infty}\frac{\log P^{\omega}(\pi,f,F_n,\epsilon)}{|F_n|}$$
exists and is independent of  the choice of  F\o lner sequences $\{F_n\}_{n\ge1}$  of $G$.

We define \emph{ $\omega$-weighted amenable topological pressure of $f$} as
$$P^{\omega}(\pi,f,G)=\lim_{\epsilon\to0}\lim_{n\to\infty}\frac{\log P^{\omega}(\pi,f,F_n,\epsilon)}{|F_n|}.$$
\end{df}

When $f=0$, we call $h_{top}^{\omega}(\pi,G):=P^{\omega}(\pi,0,G)
$ the\emph{ $\omega$-weighted amenable topological entropy }of $X$. Since  integral group $\mathbb{Z}$ is amenable group, so our definition is exactly  a   generalization of   Tsukamoto's weighted topological pressure\cite{t23}. Moreover, if  $Y$ is a singleton and $\omega =1$,   $P^{\omega}(\pi,f,G)$ is reduced to the  amenable topological pressure (or entropy):
\begin{align*}
P(f,G)&=\lim_{\epsilon\to0}\lim_{n\to\infty}\frac{\log P(X,f,F_n,\epsilon)}{|F_n|},\\
h_{top}(X,G)&=\lim_{\epsilon\to0}\lim_{n\to\infty}\frac{\log\#(X,F_n, \epsilon)}{|F_n|},
\end{align*}
where $\#(X,F, \epsilon)=P(X,0,F,\epsilon)$.

\subsection{Measure-theoretical entropy of amenable group}  To prove the lower bound of Theorem \ref{thm 1.1}, in this subsection, we collect some ingredients concerning measure-theoretic entropy of amenable group defined by open covers and partitions. 

Denote by $M(X)$, $M(X,G)$ and $E(X,G)$ the sets of all Borel probability measures on $X$ endowed with the weak*-topology,  all $G$-invariant Borel probability measures on $X$ and  all $G$-invariant ergodic Borel probability measures on $X$, respectively. The amenability of $G$ ensures
that  $E(X,G)\not=\emptyset$  and $M(X,G)$ is a compact convex subset of $M(X)$.

 Let $\PP_X$ denote the set of all finite Borel  partitions on $X$. For $\alpha,\beta\in \PP_X$, the join of $\alpha$ and $\beta$ is  $\alpha\vee\beta=\{A\cap B:A\in\alpha,~B\in\beta\}.$ Given $F\in\FF(G)$, we set  $\alpha^{F}=\vee_{g\in F}g^{-1}\alpha$,  where $g^{-1}\alpha=\{g^{-1}A:A\in\alpha\}$. One says that $\alpha$ refines $\beta$, denoted by  $\alpha \succ \beta$,   if for every  $A\in\alpha$ there exists $B\in \beta$ such that $A\subset B$. Given  $\mu\in M(X,G)$,  recall that the \emph{partition entropy of $\alpha$}, and the  \emph{condition entropy of $\alpha$  w.r.t $\beta$}  are respectively given by $$H_{\mu}(\alpha)=-\sum_{A\in\alpha}\mu(A)\log\mu(A).$$
$$H_{\mu}(\alpha|\beta)=-\sum_{B\in\beta}\mu(B)\sum_{A\in\alpha}
\mu_B(A)\log\mu_B(A),$$
where $\mu_B(\cdot)=\frac{\mu(\cdot\cap B)}{\mu(B)}$ is the conditional measure defined  on $B$, and we use the convention that $0\cdot \log 0=0$.

The \emph{measure-theoretic $\mu$-entropy of $\alpha$} is defined by 
$$h_{\mu}(G,\alpha)=\lim_{n\to\infty}\frac{1}{|F_n|}
H_{\mu}(\alpha^{F_n}),$$
where the  limit exists since  $H_{\mu}(\alpha^{F})$ is sub-additive in $F$. The \emph{measure-theoretic $\mu$-entropy of $(X,G)$} is given by 
\begin{align*}
	h_{\mu}(G, X)=\sup_{\alpha\in\PP_X}h_{\mu}(G,\alpha).
\end{align*}

Next, we  recall the equivalent definition of measure-theoretic entropy  defined by open covers. Let $\mathcal{C}_X^{o}$  denote  the set of all finite open covers of $X$. The \emph{Lebesgue number} of $\UU$, denoted by $\Leb(\mathcal{U})$,  is  the largest  positive number $\delta >0$ such that  every open ball of radius $\delta$ is contained in  some element of $\UU$.  Given $\mu\in M(X,G)$ and $\UU\in \mathcal{C}_X^{o}$,  set
$$H_{\mu}(\UU)=\inf_{\alpha\in\PP_X,\alpha\succeq\UU}H_{\mu}(\alpha).$$
Since   $H_{\mu}(\UU^{F})$ is sub-additive in $F$, then we define the \emph{measure-theoretic $\mu$-entropy of $\UU$} as
$$h_{\mu}(G,\UU)=\lim_{n\to\infty}\frac{1}{|F_n|}H_{\mu}(\UU^{F_n}).$$
Huang and Ye \cite[Theorem 3.5]{hyz11} proved that  for  all $\mu \in M(X,G)$
$$h_{\mu}(G, X)=\sup_{\UU\in\mathcal{C}_X^o}h_{\mu}(G,\UU).$$
Using above fact,   one can easily get the following:
\begin{lem}\label{lem 2.3}
Let $(X,G)$ be a $G$-system. Then for every $\mu \in M(X,G)$,
$$h_{\mu}(G,X)=\lim\limits_{\diam \UU \to 0} h_{\mu}(G,\UU).$$
\end{lem}

\subsection{Zero-dimensional extensions of amenable group actions}
In this subsection, we recall the theory of principle extension of amenable group \cite{d11,dh13,huc21}, which is used  to prove the upper bound of Theorem \ref{thm 1.1}.

Let  $\pi:(X,G)\rightarrow (Y,G)$
be a factor map and $\mu \in M(X,G)$. For $\alpha, \beta \in \PP_X$,   we define
$$h_{\mu}(G,\alpha|\beta)=\lim\limits_{n\to \infty}\frac{1}{|F_n|}H_{\mu}(\alpha^{F_n}|\beta^{F_n}).$$
Every $\beta \in \PP_Y$ can be lifted to a partition  of $X$ via $\pi$.  Put 
\begin{align*}
	h_{\mu}(G, \alpha|Y)=\inf_{\beta \in \PP_Y} h(\mu,\alpha|\pi^{-1}(\beta)).
\end{align*}
The \emph{conditional measure-theoretic entropy  of $\mu$ w.r.t. $\pi$} is defined by 
$$h_{\mu}(G, X|Y)=\sup_{\alpha \in \PP_X} 	h_{\mu}(G, \alpha|Y).	$$

\begin{df} \label{def 2.4}\cite[Definition 2.8]{huc21} The system $(X,G)$ is a principle extension of $(Y,G)$ via the factor map $\pi$ if  $h_{\mu}(G, X|Y)=0$ for every $\mu \in M(X,G)$.
\end{df}

As a direct  consequence of  Definition  \ref{def 2.4}, the measure-theoretic entropy is preserved under principle extension. 
\begin{thm}\label{thm 2.5}
Let  $\pi:(X,G)\rightarrow (Y,G)$ be  a principle  factor extension of $(Y,G)$. Then 
$$h_{\mu}(G,X)=h_{\pi_{*}{\mu}}(G,Y)$$
 for  every $\mu \in M(X,G)$.
\end{thm}
\begin{proof}
It is clear that if $h_{\pi_{*}{\mu}}(G,Y)=\infty$. If $h_{\pi_{*}{\mu}}(G,Y)<\infty$, one can show that 
$	h_{\mu}(G, X|Y)=h_{\mu}(G,X)-h_{\pi_{*}{\mu}}(G,Y)$. This implies that desired result.
\end{proof}

As  mentioned above,  $h_{\mu}(G, X|Y)=h_{\mu}(G,X)-h_{\pi_{*}{\mu}}(G,Y)$ if  $h_{\pi_{*}{\mu}}(G,Y)<\infty$.  This allows us to  extend  \cite[Corollary 6.8.9]{d11} to the following:
\begin{thm}\label{thm 2.6}
Let   $(Y,G)$  be a $G$-system  with finite topological entropy. Then   $(X,G)$ is a principle extension of $(Y,G)$ via the factor map $\pi$ if and only if  $h_{\mu}(G,X)=h_{\pi_{*}{\mu}}(G,Y)$ for  every $\mu \in M(X,G)$. 
\end{thm} 

To verify that  $\pi$ is principle factor, sometimes  it  is  more convenient to   show  the conditional entropy $h_{\text{top}}(G,X|Y)=0$  rather than checking the  condition  $h_{\mu}(G,X)=h_{\pi_{*}{\mu}}(G,Y)$.  

Recall that the \emph{topological conditional entropy} of $\pi$ \cite{y15} is given by 
$$h_{\text{top}}(G,X|Y)=\lim_{\epsilon\to0}\lim_{n\to\infty}
\frac{\sup_{y\in Y}\log\#(\pi^{-1}(y),F_n,\epsilon)}{|F_n|},$$
where   inner limit exists since  $\sup_{y\in Y}\log\#(\pi^{-1}(y),F,\epsilon)$ is  sub-additive in $F$.

\begin{thm}\label{thm 2.7}
	Let   $(Y,G)$  be a $G$-system  with finite topological entropy. Then   $(X,G)$ is a principle extension of $(Y,G)$ via the factor map $\pi$ if and only if  $h_{\text{top}}(G,X|Y)=0$. 
\end{thm} 

\begin{proof}
By Theorem \ref{thm 2.6}, it suffices to show 
$$h_{\text{top}}(G,X|Y)=\sup_{\mu \in M(X,G)}\{h_{\mu}(G,X)-h_{\pi_{*}{\mu}}(G,Y)\}.$$

By \cite[Theorem 5.9, Corollary 5.15]{y15}, one has 
$$h_{\text{top}}(G,X|Y)=\sup_{\nu\in M(Y,G)}\sup_{\mu \in M(X,G),\pi_{*}\mu=\nu}\{h_{\mu}(G,X)-h_{v}(G,Y)\}.$$

Using the fact $\pi_{*}(M(X,G))=M(Y,G)$, see [Appenix, Theorem A],  one can get the desired result.
\end{proof}

   A compact metrizable space is said to be \emph{ zero-dimensional} if clopen subsets form an open basis of the topology. The following   powerful tool developed by  Huczek  \cite[Theorem 3.2]{huc21} says that any $G$-system admits  a  zero-dimensional principle extension.

\begin{thm}\label{thm 2.8}
Let $(Y,G)$ be a $G$-system. Then there  exists   a zero-dimensional compact metric space $X$ such that  $(X,G)$ is a principle extension of $(Y,G)$.
\end{thm}

\section{Proofs of main result}
In this section, we divide the proof of Theorem  \ref{thm 1.1} into two parts.
 
\subsection{Proof of lower bound} 

\begin{lem}\label{lem 3.1}
	Let $p_1,\cdots,p_n$ be non-negative numbers with $\sum_{i=1}^n p_i=1$. Then for any real numbers $x_1,\cdots,x_n$, 
	$$\sum_{i=1}^n(-p_i\log p_i+p_ix_i)\le\log\sum_{i=1}^n e^{x_i}.$$
\end{lem}

\begin{proof}[Proof of lower bound]   Fix $0\le\omega\le1$. It suffices to show for every $\mu \in M(X,G)$,
   $$\omega h_{\mu}(G,X)+(1-\omega)h_{\pi_{*}\mu}(G,Y)+\omega\int_X fd\mu\leq P^{\omega}(\pi,f,G).$$
	Fix  $\mu\in M(X,G)$  and let $\epsilon>0$. Since  compact metric spaces $X$ and $Y$  admits the finite $\frac{\epsilon}{4}$-nets, one can choose an open cover $\UU\in\CC_X^o$ with $\diam(\UU,d)\leq \epsilon$, $\Leb(\UU)\geq\frac{\epsilon}{4}$,  and  an open cover $\VV\in\CC_Y^o$ with $\diam(\VV,d^{'})\leq\epsilon$, $\Leb(\VV)\geq \frac{\epsilon}{4}$. Assume that $\alpha\in\PP_X$ with $\alpha\succeq\UU^{F_n}$, $\beta\in\PP_Y$ with $\beta\succeq\VV^{F_n}$.
Let  $\{F_n\}_{n\ge1}$  be a  F\o lner sequence of $G$. One has
	\begin{align}\label{equ 3.1}
		\begin{split}
			&\omega H_{\mu}(\UU^{F_n})+(1-\omega)H_{\pi_{*}\mu}(\VV^{F_n})+\omega\int_X S_{F_n}fd\mu\\
			\le& \omega H_{\mu}(\alpha)+(1-\omega)H_{\pi_{*}\mu}(\beta)+\omega\int_X S_{F_n}fd\mu\\
			=&H_{\pi_{*}\mu}(\beta)+\omega(H_{\mu}(\alpha)-
			H_{\pi_{*}\mu}(\beta))+\omega\int_X S_{F_n}fd\mu\\
			\le&
			H_{\pi_{*}\mu}(\beta)+\omega H_{\mu}(\alpha|\pi^{-1}\beta)+\omega\int_X S_{F_n}fd\mu.
		\end{split}
	\end{align}
	For every $B\in\beta$, we  set
	$$\alpha_B=\{A\in\alpha: A\cap \pi^{-1}(B)\neq\emptyset\}.$$
	Then  $\pi^{-1}(B)=\cup_{A\in\alpha_B}(A\cap\pi^{-1}(B))$.
	Notice that
	\begin{align}
		\begin{split}
			\int_XS_{F_n}fd\mu&=\sum_{B\in\beta}\int_{\pi^{-1}(B)}S_{F_n}fd\mu
		\\&=\sum_{B\in\beta}\sum_{A\in\alpha_B}\int_{A\cap \pi^{-1}(B)}S_{F_n}fd\mu\\&\le \sum_{B\in\beta}\sum_{A\in\alpha_B}\mu(A\cap \pi^{-1}(B))\sup_{x\in A\cap \pi^{-1}(B)}S_{F_n}f(x)\\&= \sum_{B\in\beta}\mu(\pi^{-1}(B))\sum_{A\in\alpha_B}\frac{\mu(A\cap \pi^{-1}(B))}{\mu(\pi^{-1}(B))}\sup_{x\in A\cap \pi^{-1}(B)}S_{F_n}f(x),
		\end{split}	
	\end{align}
and 
	\begin{align}
		\begin{split}
			H_{\mu}(\alpha|\pi^{-1}\beta)&=-\sum_{B\in\beta}\mu(\pi^{-1}(B))
		\sum_{A\in\alpha}\frac{\mu(A\cap \pi^{-1}(B))}{\mu(\pi^{-1}(B))}\log\frac{\mu(A\cap \pi^{-1}(B))}{\mu(\pi^{-1}(B))}\\
		&=-\sum_{B\in\beta}\mu(\pi^{-1}(B))
		\sum_{A\in\alpha_B}\frac{\mu(A\cap \pi^{-1}(B))}{\mu(\pi^{-1}(B))}\log\frac{\mu(A\cap \pi^{-1}(B))}{\mu(\pi^{-1}(B))}.
		\end{split}	
	\end{align}
Hence, we get
	\begin{align}
		\begin{split}
		&H_{\mu}(\alpha|\pi^{-1}\beta)+\int_XS_{F_n}fd\mu\\
		\le&
		\sum_{B\in\beta}\mu(\pi^{-1}(B))\left\{\sum_{A\in\alpha_B}\frac{\mu(A\cap \pi^{-1}(B))}{\mu(\pi^{-1}(B))}(-\log \frac{\mu(A\cap \pi^{-1}(B))}{\mu(\pi^{-1}(B))}+\sup_{x\in A\cap \pi^{-1}(B)}S_{F_n}f(x))\right\}\\
		\le& \sum_{B\in\beta}\pi_{*}\mu(B)\log\sum_{A\in\alpha_B}e^{\sup_{x\in A\cap \pi^{-1}(B)}S_{F_n}f(x)},~\text {by Lemma~\ref{lem 3.1}}.
		\end{split}		
	\end{align}
This implies that
	\begin{align}
		\begin{split}
			&H_{\pi_{*}\mu}(\beta)+\omega H_{\mu}(\alpha|\pi^{-1}\beta)+\omega\int_X S_{F_n}fd\mu\\
			\le&
		-\sum_{B\in\beta}\pi_{*}\mu(B)\log\pi_{*}\mu(B)+
		\sum_{B\in\beta}\pi_{*}\mu(B)\log(\sum_{A\in\alpha_B}e^{\sup_{x\in A\cap \pi^{-1}(B)}S_{F_n}f(x)})^{\omega}\\
		\le&
		\log\sum_{B\in\beta}(\sum_{A\in\alpha_B}e^{\sup_{x\in A\cap \pi^{-1}(B)}S_{F_n}f(x)})^{\omega}, ~\text {by Lemma~\ref{lem 3.1}~ again}.
		\end{split}	
	\end{align}
Taking infimums on $\alpha$, $\beta$   for (\ref{equ 3.1}),
	\begin{align}\label{equ 3.6}
		\begin{split}
			&\omega H_{\mu}(\UU^{F_n})+(1-\omega)H_{\pi_{*}\mu}(\VV^{F_n})+\omega\int_X S_{F_n}fd\mu\\
			\le&\log\inf_{\alpha\in\PP_X, \alpha\succeq\UU^{F_n},\atop
			\beta\in\PP_Y,\beta\succeq\VV^{F_n}.}
			\{\sum_{B\in\beta}(\sum_{A\in\alpha_B}e^{\sup_{x\in A\cap \pi^{-1}(B)}S_{F_n}f(x)})^{\omega}\}\\
			\le&\log\inf_{\alpha\in\PP_X, \alpha\succeq\UU^{F_n}, \atop
			\beta\in\PP_Y,\beta\succeq\VV^{F_n},
			\alpha\succeq\pi^{-1}\beta.}
			\{\sum_{B\in\beta}(\sum_{A\in\alpha_B}e^{\sup_{x\in A\cap \pi^{-1}(B)}S_{F_n}f(x)})^{\omega}\}.
		\end{split}
		\end{align}
	
	Let $\gamma>0$.  There exists $\VV_0=\{V_i\}_{i=1}^k\in \CC_Y^o$ with $\diam(\VV_0,d_{F_n}^{'})<\frac{\epsilon}{8}$ such that
	\begin{align}\label{17}
		\sum_{i=1}^k(P(\pi^{-1}(V_i),f,F_n,\frac{\epsilon}{8}))^{\omega}<
		P^{\omega}(\pi,f,F_n,\frac{\epsilon}{8})+\frac{\gamma}{2}.
	\end{align}
	Let $B_1=V_1$,
	$B_i=V_i\setminus\cup_{t=1}^{i-1}V_t$ for $i=2,\cdots, k$. Then  $\beta=\{B_i\}_{i=1}^k \in \mathcal{P}_Y$ with  $\diam(\beta,d_{F_n}^{'})<\frac{\epsilon}{8}$. Since $\Leb(\VV)>\frac{\epsilon}{4}$, we  have $\beta\succeq\VV^{F_n}$ and 
	\begin{align}
		\begin{split}
			\sum_{i=1}^k(P(\pi^{-1}(B_i),f,F_n,\frac{\epsilon}{8}))^{\omega}
		\leq& \sum_{i=1}^k(P(\pi^{-1}(V_i),f,F_n,\frac{\epsilon}{8}))^{\omega}\\
		<& \sum_{i=1}^k(P(\pi^{-1}(V_i),f,F_n,\frac{\epsilon}{8}))^{\omega}+\frac{\gamma}{2}.
		\end{split}
	\end{align}

	So there exists $\theta>0$ such that
	\begin{align}\label{equ 3.9}
		\sum_{i=1}^k(P(\pi^{-1}(B_i),f,F_n,\frac{\epsilon}{8})+\theta)^{\omega}
		<	P^{\omega}(\pi,f,F_n,\frac{\epsilon}{8})+\gamma.
	\end{align}
	For each $i\in\{1,\cdots,k\}$, we can choose open cover $\UU^{i}=\{U^{i}_j\}_{j=1}^{m_i}$ of $X$ with $\diam(\UU^{i},d_{F_n})<\frac{\epsilon}{8}$ so that
	$$\pi^{-1}(B_i)\subset U^{i}_1\cup\cdots\cup U^{i}_{m_i}$$
and 
\begin{align}\label{equ 3.10}
\sum_{j=1}^{m_i}e^{\sup_{U^{i}_j}S_{F_n}f}<P(\pi^{-1}(B_i),f,F_n,
\frac{\epsilon}{8})+\theta.
\end{align}
	Set $A^{i}_1=U^{i}_1 \cap\pi^{-1}(B_i)$,
	$A^{i}_j=(U^{i}_j\setminus\cup_{t=1}^{j-1}U^{i}_t) \cap\pi^{-1}(B_i)$ for $j=2,\cdots, m_i$.  Then for every $1\leq i\leq k$,  $\alpha^{i}=\{{A}^{i}_j\}_{j=1}^{m_i} $  is a  finite partition of $\pi^{-1}(B_i)$ with $\diam(\alpha^{i},d_{F_n})<\frac{\epsilon}{8}$. Since $\Leb(\UU)>\frac{\epsilon}{4}$, we have  $\alpha^{i}\succeq\UU^{F_n}$ and
\begin{align}\label{equ 3.11}
\sum_{j=1}^{m_i}e^{\sup_{{A}^{i}_j}S_{F_n}f}\le\sum_{j=1}^
{m_i}e^{\sup_{U^{i}_j}S_{F_n}f}
\end{align}
	Thus, by inequalities (\ref{equ 3.9}), (\ref{equ 3.10}),  (\ref{equ 3.11}), we get
	\begin{align}
		\sum_{i=1}^k(\sum_{j=1}^{m_i}e^{\sup_{{A}^{i}_j}S_{F_n}f})^{\omega}<
		P^{\omega}(\pi,f,F_n,\frac{\epsilon}{8})+\gamma
	\end{align}
Put $\alpha=\{{A}^{i}_j: j=1,\cdots,m_i,i=1,\cdots,k\}\in \mathcal{P}_X$. Then  $\alpha\succeq\UU^{F_n}$ and $\alpha\succeq\pi^{-1}\beta$.
It follows  from inequality (\ref{equ  3.6}) that
	\begin{align}\label{equ 3.13}
		\omega H_{\mu}(\UU^{F_n})+(1-\omega)H_{\pi_{*}\mu}(\VV^{F_n})+\omega\int_X S_{F_n}fd\mu\leq  \log (P^{\omega}(\pi,f,F_n,\frac{\epsilon}{8}).
	\end{align}
	by letting $\gamma \to 0$.
  Noticing that $\int_X S_{F_n}fd\mu=|F_n|f$ and then  dividing $|F_n|$ in both side of  the inequality (\ref{equ 3.13}), we have
	\begin{align*}
 	\omega 	h_{\mu}(G,\UU)+(1-\omega) 	h_{\pi_{*}\mu}(G,\UU)+\omega\int_X fd\mu\leq \lim\limits_{n \to \infty}\frac{\log (P^{\omega}(\pi,f,F_n,\frac{\epsilon}{8}))}{|F_n|}.
	\end{align*}
Since $\epsilon >0$ is arbitrary,  by  Lemma \ref{lem 2.3} we get desired result.

\end{proof}

\subsection{Proof of upper bound}
In this subsection, we are devoted to proving    converse inequality
\begin{align*}
P^{\omega}(\pi,f,G)&\leq\sup_{\mu\in M(X,G)}\{\omega h_{\mu}(G,X)+(1-\omega)h_{\pi\mu}(G,Y)+\omega\int_X fd\mu\}\\
&:= P^{\omega}_{\text{var}}(\pi,f,G).
\end{align*}
The next  several lemmas majors the  relations of  weighted   amenable  topological pressure in terms of different factor maps. 
\begin{lem}\label{lem 3.2}
	Let $\pi: (X,G)\rightarrow (Y,G), \varphi: (X^{'}, G)\rightarrow (X,G)$ be two factor maps between $G$-systems. Let $f\in C(X,\mathbb{R})$.
	\begin{align*}
	\xymatrix{
		(X^\prime,G)\ar[r]^{\varphi}\ar[dr]^
		{\pi\circ\varphi}&(X,G)\ar[d]^{\pi}\\
		&(Y,G)}
	\end{align*}
	Then for each $0\leq \omega \leq 1$,
	$$P^{\omega}(\pi,f,G)\le P^{\omega}(\pi\circ\varphi,f\circ\varphi,G).$$

\end{lem}
\begin{proof}
	Let $\widetilde{d}$ be a metric on $X^{'}$.   Let  $\{F_n\}_{n\ge1}$  be a  F\o lner sequence of $G$.  For any $\epsilon>0$, there exists $0<\delta<\epsilon$ such that for any $x_1,x_2\in X^{'}$, one has
	$$\widetilde{d}(x_1,x_2)<\delta\implies d(\varphi(x_1),\varphi(x_2))<\epsilon.$$
	Then for every $F\in \mathcal{F}(G)$,
	$$\widetilde{d}_{F}(x_1,x_2)<\delta\implies d_{F}(\varphi(x_1),\varphi(x_2))<\epsilon.$$
 This implies that for any  $\Omega\subset X^{'}$,
	$$P(\varphi(\Omega),f,F_n,\epsilon)\le P(\Omega,f\circ\varphi,F_n,\delta).$$
Notice that for any $V\subset Y$,
	$\varphi((\pi\circ\varphi)^{-1}(V))=\pi^{-1}(V).$ Therefore,
	$$P(\pi^{-1}(V),f,F_n,\epsilon)\le P((\pi\circ\varphi)^{-1}(V),f\circ\varphi,F_n,\delta).$$
	Let $\VV=\{V_1,\cdots,V_k\}$ be an open cover of $Y$ with $\diam(\VV,d^{'}_{F_n})<\delta$. Then
	\begin{align*}
		P^{\omega}(\pi,f,F_n,\epsilon)&\le\sum_{i=1}^k(P(\pi^{-1}(V_i),
		f,F_n,\epsilon))^{\omega}\\&\le\sum_{i=1}^k
		P((\pi\circ\varphi)^{-1}(V_i),f\circ\varphi,F_n,\delta)^{\omega},
	\end{align*}
	which implies that
	$$P^{\omega}(\pi,f, F_n,\epsilon)\le P^{\omega}(\pi\circ\varphi,f\circ\varphi, F_n,\delta).$$
This completes the proof.
	
\end{proof}

\begin{lem}\label{lem 3.3}
	Let $\pi: (X,G)\rightarrow (Y,G), \phi: (Y^{'}, G)\rightarrow (Y,G)$ be  two  factor maps between $G$-systems. Define the fiber product
	$$X\times_Y Y^{'}=\{(x,y)\in X\times Y^{'}:\pi(x)=\phi(y)\},$$
	 and $(X\times_Y Y^{'},G)$ is a  $G$-system defined by $g(x,y)=(gx,gy)$. We define factor maps: $\varphi:X\times_Y Y^{'}\rightarrow X$ and $\Pi:X\times_Y Y^{'}\rightarrow Y^{'}$ by $$\varphi(x,y)=x\text{,  }\Pi(x,y)=y.$$
	 \begin{align*}
	 	\xymatrix{
	 	(X\times_{Y}Y^\prime,G)\ar[r]^-{\varphi}\ar[d]^{\Pi}&(X,G)\ar[d]^{\pi}\\
	 	(Y^\prime,G)\ar[r]^-{\phi}&(Y,G)
	 }
	 \end{align*}
	Then    for every $f\in C(X,\mathbb{R})$,
	$$P^{\omega}(\pi,f,G)\le P^{\omega}(\Pi,f\circ\varphi,G).$$
\end{lem}
\begin{proof}	
	Let $\widetilde{d}$ be a metric on $Y^{'}$. We  define a metric  $\rho$ on $X\times_Y Y^{'}$ by
	$$\rho((x_1,y_1),(x_2,y_2))=\max\{d(x_1,x_2),\widetilde{d}(y_1,y_2)\}.$$
Then 
	$$\rho((x_1,y_1),(x_2,y_2))<\epsilon\implies d(x_1,x_2)<\epsilon.$$
 Let $\{F_n\}_{n\ge1}$ be a  F\o lner sequence of $G$ and  $\Omega\subset X\times_Y Y^{'}$. We have
	$$P(\varphi(\Omega),f,F_n,\epsilon)\le P(\Omega,f\circ\varphi,F_n,\epsilon).$$
	For  each  $A\subset Y^{'}$,   $	\pi^{-1}(\phi(A))=\varphi(\Pi^{-1}(A)).$ It follows  that
	\begin{align}\label{equ 3.14}
		P(\pi^{-1}(\phi(A)),f,F_n,\epsilon)&=P(\varphi(\Pi^{-1}(A)),f,F_n,\epsilon)\\
		&\le P(\Pi^{-1}(A),f\circ\varphi,F_n,\epsilon)\nonumber.
	\end{align}
	Choose $0<\delta<\epsilon$ such that for any $y_1,y_2\in Y^{'}$ satisfying
	$$\widetilde{d}(y_1,y_2)<\delta\implies d^{'}(\phi(y_1),\phi(y_2))<\epsilon.$$
Then by (\ref{equ 3.14})
	\begin{align}\label{equ 3.15}
		P(\pi^{-1}(\phi(A)),f,F_n,\epsilon)\le P(\Pi^{-1}(A),f\circ\varphi,F_n,\delta).
	\end{align}
	 Let 
	$s>P^{\omega}(\Pi,f\circ\varphi,F_n,\delta).$
	Then there exists an open cover $Y^{'}=V_1\cup\cdots\cup V_k$ satisfying $\diam(V_i,\widetilde{d}_{F_n})<\delta$ for each $1\le i\le k$ and
	$$\sum_{i=1}^k(\Pi^{-1}(V_i),f\circ \varphi,F_n,\delta)^{\omega}<s.$$
	We can take compact subset $A_i\subset V_i$ so that $Y^{'}=A_1\cup\cdots\cup A_k$. Then 
	\begin{align*}
		\sum_{i=1}^k(\pi^{-1}(\phi(A_i)),f,F_n,\epsilon)^{\omega}&\le \sum_{i=1}^k(\Pi^{-1}(A_i),f\circ \varphi,F_n,\delta)^{\omega}\\
		&\le \sum_{i=1}^k(\Pi^{-1}(V_i),f\circ \varphi,F_n,\delta)^{\omega}<s.
	\end{align*}
	Each $\phi(A_i)$ is a closed subset of $Y$(since $Y$ is a Hausdorff space) with $\diam(\phi(A_i),d^{'}_{F_n})<\epsilon$. Then there exist open subsets $W_i\supset \phi(A_i)$ of $Y$ for $1\le i\le k$ such that $\diam(W_i,d^{'}_{F_n})\le\epsilon$ and
	$\sum_{i=1}^k(\pi^{-1}(W_i),f,F_n,\epsilon)^{\omega}<s.$ This shows that
	$P^{\omega}(\pi,f,F_n,\epsilon)<s.$
	Letting $s\to P^{\omega}(\Pi,f\circ\varphi,F_n,\delta)$ and then dividing $|F_n|$, we get 
	$$P^{\omega}(\pi,f,G)\le P^{\omega}(\Pi,f\circ\varphi,G).$$
\end{proof}

\begin{lem}\label{lem 3.4}
	Under the setting of Lemma \ref{lem 3.3}, suppose that  $h_{top}(X,G)<\infty$.  if $\phi:Y^{'}\rightarrow Y$ is a principal factor map,
	\begin{align*}
		\xymatrix{
		(X\times_{Y}Y^\prime,G)\ar[r]^-{\varphi}\ar[d]^{\Pi}&(X,G)\ar[d]^{\pi}\\
		(Y^\prime,G)\ar[r]^-{\phi:principal}&(Y,G)}
	\end{align*}
	then $\varphi$ is  also a principal factor map.
\end{lem}
\begin{proof}

	Let $\{F_n\}_{n\ge1}$ be a F\o lner sequence of $G$. For each $x\in X$ and $n\in \mathbb{N}$,  clearly $\Pi$ is an isometric mapping between $(\varphi^{-1}(x),\rho_{F_n})$ and $(\phi^{-1}(\pi(x),\widetilde{d}_{F_n})$. This shows  for any $\epsilon>0$,
	$$\#(\varphi^{-1}(x),F_n,\epsilon)=\#
	(\phi^{-1}(\pi(x)),F_n,\epsilon).$$
Notice that
	$\sup_{x\in X}\#(\phi^{-1}(\pi(x)),F_n,\epsilon)=\sup_{y\in Y}\#(\phi^{-1}(y),F_n,\epsilon).$
	So
	$$\sup_{x\in X}\#(\varphi^{-1}(x),F_n,\epsilon)= \sup_{y\in Y}\#(\phi^{-1}(y),F_n,\epsilon).$$
Since $\phi$ is a principal map and $h_{top}(G,Y)<\infty$, we know that
	$h_{\text{top}}(G,X\times_Y Y^{'}|X)=0$ by  Theorem \ref{thm 2.7}.
	So $\varphi$ is principal by   using Theorem \ref{thm 2.7} again.
	
\end{proof}

\begin{lem}\label{lem 3.5}
	Let $\pi:(X,G)\rightarrow (Y,G)$ be a factor map between two $G$-systems, $h_{top}(X,G)<\infty$ and let $f\in C(X,\mathbb{R})$. Then there exist $G$-systems $(X^{'},G)$, $(Y^{'},G)$ with the factor map $\pi^{'}:(X^{'},G)\rightarrow (Y^{'},G)$ and $f^{'}\in C(X^{'}, \mathbb{R})$ so that \\
	(1) Both $X^{'}$ and $Y^{'}$ are zero dimensional;\\
	(2) For each $0\le\omega\le1$, one has
	$$P^{\omega}(\pi,f,G)\le P^{\omega}(\pi^{'},f^{'},G)$$
	and
	$$P^{\omega}_{\text{var}}(\pi^{'},f^{'},G)
	\le P^{\omega}_{\text{var}}(\pi,f,G).$$
\end{lem}
\begin{proof}
	By Theorem \ref{thm 2.8}, there exists a zero-dimensional principal extension via factor map
	$\phi:(Y^{'},G)\rightarrow (Y,G)$.
	We define  the fiber product $X\times_Y Y^{'}$, the projections
	$\varphi:X\times_Y Y^{'}\rightarrow X$ and $\Pi:X\times_Y Y^{'}\rightarrow Y^{'}$ as Lemma \ref{lem 3.3}.
\begin{align*}
	\xymatrix{
	(X\times_{Y}Y^\prime,G)\ar[r]^-{\varphi}\ar[d]^{\Pi}&(X,G)\ar[d]^{\pi}\\
	(Y^\prime,G)\ar[r]^-{\phi:principal}&(Y,G)
}
\end{align*}
	Then  $\varphi$ is a principal factor map by Lemma \ref{lem 3.4}.
	
	By Lemma \ref{lem 3.3}, for any $0<\omega<1$
	\begin{align}\label{equ 3.16}
	P^{\omega}(\pi,f,G)\le P^{\omega}(\Pi,f\circ\varphi,G).
	\end{align}
Then by Theorem \ref{thm 2.5}, for every $\mu\in M(X\times_Y Y^{'},G)$, 
	$$h_{\mu}(G,X\times_Y Y^{'})=h_{\varphi_{*}\mu}(X,G)$$
	and
	$$h_{\Pi_{*}\mu}(G,Y^{'})=h_{\phi_{*}(\Pi_{*}\mu)}(G,Y)=
	h_{\pi_{*}(\varphi_{*}\mu)}(G,Y).$$
	Then
	\begin{align}\label{equ 3.17}
	\begin{split}
		&P^{\omega}_{\text{var}}(\Pi,f\circ\varphi,G)\\	
	=&\sup\{\omega h_{\mu}(G,X\times_Y Y^{'})+(1-\omega)h_{\Pi_{*}\mu}(G,Y^{'})\\
	&+\omega\int_{X\times_Y Y^{'}} f\circ \varphi d\mu:{\mu\in M(X\times_Y Y^{'},G)}\}
	\\
	=&\sup\{\omega h_{\varphi_{*}\mu}(G,X )+(1-\omega)h_{\pi_{*}(\varphi_{*}\mu)}(G,Y)\\
	&+\omega\int_{X} fd\varphi_{*}\mu:{\mu\in M(X\times_Y Y^{'},G)}\}
	\\
	\le& P^{\omega}_{\text{var}}(\pi,f,G).
	\end{split}
	\end{align}
	
	Applying Theorem \ref{thm 2.8} again to the system $(X\times_Y Y^{'},G)$,  we get a zero-dimension principal extension via factor map
	$\psi:(X^{'},G)\rightarrow(X\times_Y Y^{'},G).$
\begin{align*}
	\xymatrix{
	(X^\prime,G)\ar[r]^-{\psi}\ar[dr]_
	-{\Pi\circ\psi}&(X\times_{Y}Y^\prime,G)\ar[r]^-{\varphi}\ar[d]_{\Pi}&(X,G)\ar[d]^{\pi}\\
	&(Y^\prime,G)\ar[r]^-{\phi}&(Y,G)
}
\end{align*}
	By Lemma \ref{lem 3.2}, 
	\begin{align}\label{equ 3.18}
		P^{\omega}(\Pi,f\circ\varphi,G)\le P^{\omega}(\Pi\circ\psi,f\circ\varphi\circ\psi,G).
	\end{align}
Moreover, we have 
	\begin{align}\label{equ 3.19}
	\begin{split}
			&P^{\omega}_{\text{var}}(\Pi\circ\psi,f\circ\varphi\circ\psi,
		G)\\
		=&\sup\{\omega h_{\mu}(G,X^{'})+(1-\omega) h_{\Pi_{*}(\psi_{*}\mu)}(G,Y^{'})\\
		&+\omega \int_{X^{'}}f\circ\varphi\circ\psi d\mu:\mu\in M(X^{'},G)\}\\
		=&\sup\{\omega h_{\psi_{*}\mu}(G,X\times_Y Y^{'})+(1-\omega) h_{\Pi_{*}(\psi_{*}\mu)}(G,Y^{'})\\
		&+\omega \int_{X\times_Y Y^{'}}f\circ\varphi d\psi_{*}\mu:\mu\in M(X^{'},G)\}
		\\
		\le& P^{\omega}_{\text{var}}(\Pi,f\circ\varphi,G).
	\end{split}
	\end{align}
	
	It follows from (\ref{equ 3.16}), (\ref{equ 3.18}) that
	$$P^{\omega}(\pi,f,G)\le P^{\omega}(\Pi\circ\psi,f\circ\varphi\circ\psi,G)$$
	and  (\ref{equ 3.17}), (\ref{equ 3.19}) that
	$$P^{\omega}_{\text{var}}(\Pi\circ\psi,f\circ\varphi\circ\psi,G)
	\le P^{\omega}_{\text{var}}(\pi,f,G).$$
	Then  $\pi^{'}=\Pi\circ\psi :(X^{'},G)\rightarrow (Y^{'},G)$
	and $f^{'}=f\circ\varphi\circ\psi:X^{'}\rightarrow\R$ are what we need.
	
\end{proof}

If  $h_{top}(X,G)=\infty$, then the maximal entropy measure  of $X$ is not empty (i.e. there exists $\mu \in M(X,G)$ so that $h_{\mu}(G,X)=\infty$) by variational principle for topological entropy\cite[Theorem 9.48]{kl16}). In this case,  
$$P^{\omega}(\pi,f,G)\leq P^{\omega}_{\text{var}}(\pi,f,G)=\infty$$
 naturally holds.  Hence  we  only need to show\begin{align*}
 	P^{\omega}(\pi,f,G)&\leq\sup_{\mu\in M(X,G)}\{\omega h_{\mu}(G,X)+(1-\omega)h_{\pi\mu}(G,Y)+\omega\int_X fd\mu\}.
 \end{align*} 
  for the case  $h_{top}(X,G)<\infty$. However, by Lemma \ref{lem 3.5}  the proof is reduced to  the zero dimensional case.
\begin{prop}
	Let $\pi:(X,G)\rightarrow (Y,G)$ be the  factor map between two zero dimensional  $G$-systems. Let $f\in C(X,\mathbb{R})$. Then for each $0\le\omega\le1$,
	$$P^{\omega}(\pi,f,G)\le P^{\omega}_{\text{var}}(\pi,f,G).$$
\end{prop}
\begin{proof}
 We show  that for each $\epsilon >0$ there exists $\mu\in M(X,G)$ satisfying
	$$\lim_{n\to\infty}\frac{\log P^{\omega}(\pi,f,F_n,\epsilon)}{|F_n|} \leq \omega h_{\mu}(G,X)+(1-\omega)h_{\pi_{*}\mu}(G,Y)+\omega\int_X fd\mu.$$
	
	Let $Y=A_1\cup\cdots\cup A_{\alpha}$ be a clopen partition with $\diam(A_a,d^{'})<\epsilon$ for all $1\le a\le \alpha$. From $\dim X=0$, for each $1\le a\le \alpha$, we can choose a clopen partition so that
	$$\pi^{-1}(A_a)=\cup_{b=1}^{\beta_a}B_{ab}$$
	with $\diam (B_{ab},d)<\epsilon$ for all $1\le b\le \beta_a$. Set $\AA=\{A_1,\cdots,A_{\alpha}\}$ and $\BB=\{B_{ab}:1\le a\le \alpha,1\le b\le \beta_a\}$.  Then  $\BB\succ \pi^{-1}(\AA)$.
	
	Let $\{F_n\}_{n\ge1}$ be a F\o lner sequence of $G$. For each $n\ge1$, $\BB^{F_n}\succ \pi^{-1}(\AA^{F_n})$.  For $\emptyset\neq A\in\AA^{F_n}$, put
	$$\BB^{F_n}_A:=\{B\in \BB^{F_n}:B\cap\pi^{-1}(A)\neq\emptyset\}=\{B\in \BB^{F_n}:B\subset\pi^{-1}(A),B\neq\emptyset\}.$$
    Namely,
	$\pi^{-1}(A)=\cup_{B\in \BB^{F_n}_A}B.$
	We set
	$$Z_{F_n,A}=\sum_{B\in \BB^{F_n}_A}e^{\sup_B S_{F_n}f}$$
   and
	$$Z_{F_n}=\sum_{\emptyset\not= A\in \AA^{F_n}}(Z_{F_n,A})^{\omega}.$$	This shows that $	P^{\omega}(\pi,f,F_n,\epsilon)\le Z_{F_n}.$
   For  each  $\emptyset \not= B\in \BB^{F_n}$, by $\AA^{F_n}(B)$  we denote the unique element of $\AA^{F_n}$ containing $\pi(B)$.  Then for each   $\emptyset \not= A\in \AA^{F_n}$, $\AA^{F_n}(B)=A$ for all $B\in \BB^{F_n}_A$. For each  $\emptyset \not =B\in \BB^{F_n}$, noticing that $B$ is closed,  we take a point $x_B\in B$ satisfying $S_{F_n}f(x_B)=\sup_{x\in B}S_{F_n}f(x).$ Define a probability measure on $X$ by
	\begin{align*}
		\delta_{n}&=\frac{1}{Z_{F_n}}\sum_{B\in \BB^{F_n}}(Z_{F_n,\AA^{F_n}(B)})^{\omega-1}e^{S_{F_n}f(x_B)}
		\cdot\delta_{x_B}\\&=\frac{1}{Z_{F_n}}\sum_{A\in \AA^{F_n}}\sum_{B\in \BB^{F_n}_A}
		(Z_{F_n,A})^{\omega-1}e^{S_{F_n}f(x_B)}\cdot\delta_{x_B},
	\end{align*}
where $\delta_{x}$ is the delta probability measure at the point ${x}$.  We set
	$$\mu_{n}=\frac{1}{|F_n|}\sum_{g\in F_n}\delta_{n}\circ g^{-1}.$$
	 Without loss of generality, we assume that $\mu_n \rightarrow \mu \in M(X,G)$ in the  weak* topology.
  Next, we show 
	$$\lim_{n\to\infty}\frac{\log Z_{F_n}}{|F_n|}\leq \omega h_{\mu}(G,X)+(1-\omega)h_{\pi_{*}\mu}(G,Y)+\omega\int_X fd\mu.$$
	\begin{claim}\label{claim 3.7}
		$$\omega H_{\delta_{n}}(\BB^{F_n})+(1-\omega)H_{\pi_{*}\delta_{n}}
		(\AA^{F_n})+\omega\int_XS_{F_n}fd\delta_{n}=\log Z_{F_n}.$$
	\end{claim}
	\begin{proof}[Proof of {Claim 3.7}]
	Since
		$$\pi_{*}\delta_{n}=\frac{1}{Z_{F_n}}\sum_{B\in\BB^{F_n}}(Z_{F_n,
			\AA^{F_n}(B)})^{\omega-1}e^{S_{F_n}f(x_B)}
		\cdot\delta_{\pi(x_B)},$$
		then for  each  $\emptyset \not =A\in \AA^{F_n}$,
		\begin{align*}
		\pi_{*}\delta_{n}(A)&
		=\frac{1}{Z_{F_n}}\sum_{B\in\BB^{F_n}}(Z_{F_n,
		\AA^{F_n}(B)})^{\omega-1}e^{S_{F_n}f(x_B)}\\
		&=\frac{1}{Z_{F_n}}(Z_{F_n,A})^{\omega}.
		\end{align*}
		Then
		\begin{align}\label{equ 3.20}
			H_{\pi_{*}\delta_{n}}(\AA^{F_n})=\log Z_{F_n}-\omega\sum_{A\in\AA^{F_n}}\frac{(Z_{F_n,A})^{\omega}}
			{Z_{F_n}}\log Z_{F_n,A}.
		\end{align}
		For each  $\emptyset \not= B\in\BB^{F_n}$,
		$$\delta_{n}(B)=\frac{(Z_{F_n,\AA^{F_n}(B)})^{\omega-1}}
		{Z_{F_n}}e^{S_{F_n}f(x_B)}$$
		Then
		\begin{align*}
			H_{\delta_{n}}(\BB^{F_n})&=-\sum_{B\in\BB^{F_n}}\frac
			{(Z_{F_n,\AA^{F_n}(B)})^{\omega-1}}
			{Z_{F_n}}e^{S_{F_n}f(x_B)}\cdot \log \frac{(Z_{F_n,\AA^{F_n}(B)})^{\omega}}
			{Z_{F_n}}e^{S_{F_n}f(x_B)}\\&=\frac{\log Z_{F_n}}{Z_{F_n}}\underbrace{\sum_{B\in\BB^{F_n}}(Z_{F_n,\AA^{F_n}(B)})
				^{\omega-1}e^{S_{F_n}f(x_B)}}_{(I)}\\&-\frac{\omega-1}{Z_{F_n}}
			\underbrace{\sum_{B\in\BB^{F_n}}(Z_{F_n,\AA^{F_n}(B)})
				^{\omega-1}e^{S_{F_n}f(x_B)}\log Z_{F_n,\AA^{F_n}(B)}}_{(II)}\\&-\underbrace{\sum_{B\in\BB^{F_n}}\frac{(Z_{F_n,\AA^{F_n}(B)})
					^{\omega-1}}{Z_{F_n}}e^{S_{F_n}f(x_B)}S_{F_n}f(x_B)}_{(III)}
		\end{align*}
		For the term $(I)$,
		\begin{align*}
			(I)&=\sum_{A\in\AA^{F_n}}\sum_{B\in\BB^{F_n}_A}(Z_{F_n,\AA^{F_n}(B)})
			^{\omega-1}e^{S_{F_n}f(x_B)}\\
			&=\sum_{A\in\AA^{F_n}}(Z_{F_n,A})
			^{\omega-1}Z_{F_n,A}\\
			&=Z_{F_n}.
		\end{align*}
		For the  term $(II)$,
		\begin{align*}
			(II)&=\sum_{A\in\AA^{F_n}}\sum_{B\in\BB^{F_n}_A}(Z_{F_n,A})
			^{\omega-1}e^{S_{F_n}f(x_B)}\log Z_{F_n,A}\\
			&=\sum_{A\in\AA^{F_n}}(Z_{F_n,A})
			^{\omega}\log Z_{F_n,A}.
		\end{align*}
		For the term $(III)$,
		$$\int_XS_{F_n}fd\delta_{n}=\frac{1}{Z_{F_n}}\sum_{B\in\BB^{F_n}}(Z_{F_n,\AA^{F_n}(B)})
		^{\omega-1}e^{S_{F_n}f(x_B)}S_{F_n}f(x_B)=(III).$$
		Thus
		\begin{align}\label{equ 3.21}
		H_{\delta_{n}}(\BB^{F_n})+\int_XS_{F_n}fd\delta_{n}=\log Z_{F_n}-\frac{\omega-1}{Z_{F_n}}\sum_{A\in\AA^{F_n}}{(Z_{F_n,A})^{\omega}}
		\log Z_{F_n,A}.
		\end{align}
	Combining (\ref{equ 3.20}) with (\ref{equ 3.21}), this yields that 
		$$\omega H_{\delta_{n}}(\BB^{F_n})+(1-\omega)H_{\pi_{*}\delta_{n}}
		(\AA^{F_n})+\omega\int_XS_{F_n}fd\delta_{n}=\log Z_{F_n}.$$
	\end{proof}
	By   \cite[Lemma 3.1, (3)]{hyz11}, for any $H\in \FF(G)$, 
	\begin{align}\label{equ 3.22}
	H_{\delta_{n}}(\BB^{F_n})\le \frac{1}{|H|}\sum_{g\in F_n}H_{\delta_{n}\circ g^{-1}}(\BB^H)+|F_n\setminus\{g\in G:H^{-1}g\in F_n\}|\cdot\log\#\BB
\end{align}
	and 
	\begin{align}\label{equ 3.23}
	H_{\pi_{*}\delta_{n}}(\AA^{F_n})\le \frac{1}{|H|}\sum_{g\in F_n}H_{\pi_{*}\delta_{n}\circ g^{-1}}(\AA^{F_n})+|F_n\setminus\{g\in G:H^{-1}g\in F_n\}|\cdot\log\#\AA.
	\end{align}
	Notice that $F_n\setminus\{g\in G:H^{-1}g\in F_n\} \subset \partial_{\widetilde{H}}(F_n)$, where $\widetilde{H}=H\cup\{e_G\}$. Thus by \textbf{Claim} \ref{claim 3.7}, we get
	\begin{align*}
		\frac{1}{|F_n|}\log Z_{F_n}
		&=\frac{\omega}{|F_n|} H_{\delta_{n}}(\BB^{F_n})+\frac{(1-\omega)}{|F_n|}H_{\pi_{*}\delta_{n}}
		(\AA^{F_n})+\frac{\omega}{|F_n|}\int_XS_{F_n}fd\delta_{n}
		\\
		&\le \frac{\omega}{|H|}H_{\mu_{n}}(\BB^{H})+\frac{1-\omega}{|H|}
		H_{\pi_{*}\mu_{n}}(\AA^{H})+\omega \int_X fd\mu_{n}\\
		&+\frac{\omega}{|F_n|}|
	 \partial_{\widetilde{H}}(F_n)|\cdot\log\#\BB+\frac{1-\omega}{|F_n|}|
		\partial_{\widetilde{H}}(F_n)|\cdot\log\#\AA, \text{by (\ref{equ 3.22})  and (\ref{equ 3.23})}.
	\end{align*}
  Since the partitions $\mathcal{A}$ and $\mathcal{B}$ have   empty boundary, letting $n\to\infty$,  we have $$\lim_{n \to \infty}\frac{\omega}{|F_n|}|
	\alpha_{\widetilde{H}}(F_n)|\cdot\log\#\BB+\frac{1-\omega}{|F_n|}|
	\alpha_{\widetilde{H}}(F_n)|\cdot\log\#\AA=0$$
 and hence
	\begin{align*}
		&\lim_{n\to\infty}\frac{\log P^{\omega}(\pi,f,F_n,\epsilon)}{|F_n|}\le\limsup_{n\to\infty}\frac{\log Z_{F_n}}{|F_n|}\\
		\le& \frac{\omega}{|H|}H_{\mu}(\BB^{H})+\frac{1-\omega}{|H|}
		H_{\pi_{*}\mu}(\AA^{H})+\omega \int_X fd\mu.
	\end{align*}
By the arbitrariness of the $H\in\FF(G)$, we complete the proof.	
\end{proof}

\section{Topological pressure determines invariant measures}
In this section, we  show how weighted amenable topological pressure determines the invariant measures and its measure-theoretic entropies.

\begin{prop}\label{prop 4.1}
	Let $\pi:(X,G)\rightarrow (Y,G)$ be a factor map between two $G$-systems. Let $f,h\in C(X,\mathbb{R})$, and $\omega\in (0,1]$. Then  the following  statements hold. 
	\begin{enumerate}
		\item $P^{\omega}(\pi,0,G)=h^\omega_{top}(\pi,G)$.
		\item If $f\leq h$, then  $P^{\omega}(\pi,f,G)\leq P^{\omega}(\pi,h,G)$. Particularly, $h^\omega_{top}(\pi,G)+\omega\inf f\leq P^{\omega}(\pi,f,G)\leq h^\omega_{top}(\pi,G)+\omega\sup f$.
		\item  $P^{\omega}(\pi,\cdot,G)$ is either finite or constantly $\infty$.
		\item If  $P^{\omega}(\pi,\cdot,G)\in \mathbb{R}$ for all  $f \in C(X,\mathbb{R}$, then $|P^{\omega}(\pi,f,G)-P^{\omega}(\pi,h,G)|\leq \omega||f-g||_{\infty}$ and $P^{\omega}(\pi,\cdot,G)$ is convex.
		\item  $P^{\omega}(\pi,f+c,G)=P^{\omega}(\pi,f,G)+\omega c$ for any $c\in \mathbb{R}$.
		\item $P^{\omega}(\pi,f+h\circ g-h,G)=P^{\omega}(\pi,f,G)$.
		\item $P^{\omega}(\pi,f+h,G)\leq P^{\omega}(\pi,f,G)+P^{\omega}(\pi,h,G)$.
	\end{enumerate}
\end{prop}
\begin{proof}
	It can be proved by using  the variational principle for weighted amenable topological pressure established in Theorem  \ref{thm 1.1}.  
\end{proof}

 For simplifying the notations, we call   $h^{\omega}_{\mu}(\pi,G):=\omega h_{\mu}(G,X)+(1-\omega)h_{\pi_{*}\mu}(G,Y)$ \emph{weighted  amenable measure-theoretic entropy}.  Recall that a finite signed measure on $X$ is a set function $\mu:\mathcal{B}(X)\rightarrow \mathbb{R}$ which is countably additive.

 \begin{thm}
 Let $\pi:(X,G)\rightarrow (Y,G)$ be a factor   map between two $G$-systems with $h_{top}^{\omega}(\pi,G)<\infty$. Let $\omega \in (0,1]$ and  $\mu:\mathcal{B}(X)\rightarrow \mathbb{R}$ be a finite signed measure. Then $\mu\in M(X,G)$ if and only if  $\omega\int f d\mu\leq P^{\omega}(\pi,f,G)$ for all $ f\in C(X,\mathbb{R}).$	
 \end{thm}
\begin{proof}
	If $\mu\in M(X,G)$, then one has  $\omega\int f d\mu\leq P^{\omega}(\pi,f,G)$ by Theorem \ref{thm 1.1}.
	
	Assume that   $\mu$ is a finite signed measure so that  $\omega\int f d\mu\leq P^{\omega}(\pi,f,G)$ for all $f\in C(X,\mathbb{R})$. Let  $f\geq 0$. For $\epsilon>0$ and $n>0$,  we have 
\begin{align*}
	\omega\int n(f+\epsilon)d\mu&=-\omega\int -n(f+\epsilon)d\mu\geq -P^{\omega}(\pi,-n(f+\epsilon),G)\\
	&\geq -\left\lbrace h_{top}^{\omega}(\pi,G)+\omega \sup(-n(f+\epsilon))\right\rbrace\\ &=-h_{top}^{\omega}(\pi,G)+\omega n\inf(f+\epsilon)>0
\end{align*}
for sufficiently  large $n$. This implies that
 $\int (f+\epsilon)d\mu>0$ and hence  $\int f d\mu\geq 0$. Hence, $\mu$ is a measure.  If $n\in \mathbb{Z}$,  then $\omega\int n d\mu\leq P^{\omega}(\pi,n,G)=h_{top}^{\omega}(\pi,G)+\omega n$. One  has $\mu(X)\leq 1+\frac{h_{top}^{\omega}(\pi,G)}{\omega n}$ if $n>0$ and hence $\mu(X)\leq 1$, and $\mu(X)\geq 1+\frac{h_{top}^{\omega}(\pi,G)}{\omega n}$	if $n<0$ and hence $\mu(X)\geq 1$. This shows $\mu \in M(X)$.  If $n\in \mathbb{Z}$	and $f\in C(X,\mathbb{R})$, then $\omega n\int (f\circ g-f)d\mu\leq P^{\omega}(\pi,n(f\circ g-f),G)=h_{top}^{\omega}(\pi,G)$   for all $ g\in G$. If $n\geq1$,  then  $\int (f\circ g-f)d\mu\leq0$. If $n\leq-1$,  then  $\int (f\circ g-f)d\mu\geq0$. This  implies that $\mu\in M(X,G)$.

\end{proof}	

The following lemma  is the well-known separation theorem of convex sets  given in \cite{ds58}.

\begin{lem}\label{lem 4.3}
	If $K_1,K_2$ are disjoint closed convex subsets of a locally convex linear
	topological space $V$ and  $K_1$ is compact, then there exists a continuous real-valued linear functional $F$ on $V$ such that $F(x)<F(y)$ for all $x\in K_1,y\in K_2$.
\end{lem}

	\begin{thm}\label{th4.4}
	Let  $\pi:(X,G)\rightarrow (Y,G)$ be a factor map between two $G$-systems with $h_{top}^{\omega}(\pi,G)<\infty$
    and  $\mu_0\in M(X,G)$. Then $$h^{\omega}_{\mu_0}(\pi,G) =\inf\lbrace P^{\omega}(\pi,f,G)-\omega\int f d\mu_0: f\in C(X,\mathbb{R})\rbrace $$ if and only if the entropy map $\mu\in M(X,G)\mapsto h^{\omega}_{\mu}(\pi,G)$ is upper semi-continuous at $\mu_0$.
\end{thm}

\begin{proof}
	Suppose  that $h^{\omega}_{\mu_0}(\pi,G) =\inf\lbrace P^{\omega}(\pi,f,G)-\omega\int f d\mu_0|f\in C(X,\mathbb{R})\rbrace $. Given  $\varepsilon>0$, choose $h\in C(X,\mathbb{R})$ such that $P^{\omega}(\pi,h,G)-\omega\int h d\mu_0 < h^{\omega}_{\mu_0}(\pi,G) +\frac{\varepsilon}{2}$. Let $V_{\mu_0}(h,\frac{\varepsilon}{2})=\{\mu\in M(X,G): |\int hd\mu -\int h d\mu_0|\leq \frac{\epsilon}{2\omega}\}$.  If $\mu \in V_{\mu_0}(h,\frac{\varepsilon}{2})$,  then 
\begin{align*}
	h^{\omega}_{\mu}(\pi,G)&\leq P^{\omega}(\pi,h,G)-\omega\int h d\mu,\text{by Theorem \ref{thm 1.1}}\\
	&<P^{\omega}(\pi,h,G)-\omega\int h d\mu_0 +\frac{\varepsilon}{2}\\
	&<h^{\omega}_{\mu_0}(\pi,G)+\varepsilon,
\end{align*}
which shows that  the entropy map is upper semi-continuous at $\mu_0$.

Assume that  the entropy map is upper semi-continuous at $\mu_0$. Following \cite[Theorem 8.1]{w82}, one  has $h_{\mu}(G,X)$ is affine in $\mu$ and so is $h^{\omega}_{\mu}(\pi,G)$.  By Theorem \ref{thm 1.1},  we have $$h^{\omega}_{\mu_0}(\pi,G) \leq\inf\lbrace P^{\omega}(\pi,f,G)-\omega\int f d\mu_0: f\in C(X,\mathbb{R})\rbrace.$$  Let $b>h^{\omega}_{\mu_0}(\pi,G)$ and let $C=\lbrace (\mu,t)\in M(X,G)\times\mathbb{R}: 0\leq t\leq h^{\omega}_{\mu}(\pi,G)\rbrace$. Then $C$ is a convex set since the entropy function $\mu\in M(X,G)\mapsto h^{\omega}_{\mu}(\pi,G)$ is affine. Regard  $C$ as a subset of $C(X,\mathbb{R})^\ast\times\mathbb{R}$, where  $C(X,\mathbb{R})^\ast$ is endowed with weak star topology. Then  $(\mu_0,b)\notin \overline{C}$ by the upper semi-continuity of the entropy map at $\mu_0$. Applying Lemma \ref{lem 4.3}  to the disjoint convex sets $\overline{C}$ and $(\mu_0,b)$,  there is a continuous linear functional $F:C(X,\mathbb{R})^\ast\times\mathbb{R}\rightarrow\mathbb{R}$ such that $F((\mu,t))<F((\mu_0,b))$ for all $(\mu,t)\in \overline{C}$. We know that $F$ must have  the form $F(\mu,t)=\omega\int f d\mu +td$ for some $f\in C(X,\mathbb{R})$ and some $d\in \mathbb{R}$. So  $\omega\int f d\mu_0+dh^w_{\mu_0}(\pi,T)<\omega\int f d\mu_0+db$. This shows that  $d>0$. Hence 
\begin{align*}
	h^\omega_{\mu}(\pi,G)+\omega\int\frac{f}{d} d\mu <b +\omega\int\frac{f}{d} d\mu_0
\end{align*}
for all  $\mu\in M(X,G)$. By Theorem \ref{thm 1.1}, we have $	P^{\omega}(\pi,\frac{f}{d},G)<b+ \omega\int\frac{f}{d} d\mu_0.$ This gives us 
\begin{align*}
	b>P^{\omega}(\pi,\frac{f}{d},G)-\omega\int\frac{f}{d} d\mu_0\geq \inf\lbrace P^{\omega}(\pi,h,G)-\omega\int h d\mu_0|h\in C(X,\mathbb{R})\rbrace.
\end{align*}
This competes the proof.	
	
\end{proof}

\section{Equilibrium states}
The classical variational principle  for topological entropy  provides a natural way  of selecting  the maximal entropy measures whose measure-theoretic entropy is equal to topological entropy. In  statistical mechanics, a corresponding role called thermodynamic equilibrium states is considered to 
characterize  the minimum of a thermodynamic potential(such as pressure, temperature). In  this section, we   gives some
elementary properties  and characterizations about the   equilibrium states  of weighted amenable topological pressure.

\begin{df}
	Let  $\pi:(X,G)\rightarrow (Y,G)$ be a factor map between two $G$-systems, $\omega \in (0,1]$ and  $f\in C(X,\mathbb{R})$.
	A measure $\mu\in M(X,G)$ is called an equilibrium state for $f$ if $$ P^{\omega}(\pi,f,G)=h^\omega_{\mu}(\pi,G)+\omega\int f d\mu.$$ Let $M_{f}(X,G)$ denote the collection of all equilibrium states for $f$.
\end{df}
  
\begin{thm}
Let  $\pi:(X,G)\rightarrow (Y,G)$ be a factor map between two $G$-systems, $\omega \in (0,1]$ and  $f\in C(X,\mathbb{R})$. Then 
\begin{enumerate}
	\item $M_{f}(X,G)$ is convex.
	\item If $h^\omega_{top}(\pi,G)<\infty$, then  the extreme points of $M_{f}(X,G)$ are precisely the ergodic members of $M_{f}(X,G)$.
	\item If $h^\omega_{top}(\pi,G)<\infty$ and $M_{f}(X,G)\neq\emptyset$,  then $M_{f}(X,G)$ contains an ergodic measure.
	\item If the entropy map is upper semi-continuous,  then $M_{f}(X,G)$ is compact and non-empty.
	\item If $f,h\in C(X,\mathbb{R})$ and if there exists $c\in \mathbb{R}$ such that $f-h-c$ belongs to the closure of the set $\left\lbrace b\circ g-b|b\in C(X,\mathbb{R}),g\in G\right\rbrace $ in $C(X,\mathbb{R})$, then $M_f(X,G)=M_h(X,G)$.
\end{enumerate}
\end{thm}  
 \begin{proof}
 	Let  $\mu \in M(X,G)$ and $\mu=\int_{E(X,G)}md\tau(m)$ be  the ergodic decomposition  of $\mu$ and $\mathcal{U} \in C_X^o$.  Then by \cite[Theorem 3.13]{hyz11}
 	$h_\mu(G,\mathcal{U})=\int_{E(X,G)}h_m(G,\mathcal{U})d\tau(m).$
 	Choose a sequence of finite open covers $\{\mathcal{U}_n\}_{n\geq 1}$ of  $X$ with $\mathcal{U}_{n+1}\succ \mathcal{U}_n$ and $\diam \mathcal{U}_n \to 0$ as $n$ tends $\infty$. Then by Lemma \ref{lem 2.3} $\lim_{n \to \infty}h_{\mu}(G,\mathcal{U}_n)=h_{\mu}(G,X)$ for all $\mu \in M(X,G)$. Using the monotone convergence theorem, we  have
 	$h_\mu(G,X)=\int_{E(X,G)}h_m(G,X)d\tau(m).$ Since $\pi_{*}\mu=\int_{E(Y,G)}\theta d\pi_{*}\tau(\theta)$ by [Appendix, Theorem A], so
 	\begin{align*}
 	h^{\omega}_{\mu}(\pi,G)&=\omega \int_{E(X,G)}h_m(G,X)d\tau(m) +(1-\omega) \int_{E(Y,G)}h_\theta(G,Y)d\pi_{*}\tau(\theta)\\
 	&=\omega \int_{E(X,G)}h_m(G,X)d\tau(m) +(1-\omega) \int_{E(X,G)}h_{{\pi_{*}}m}(G,Y)d\tau(m)\\
 	&=\int h^{\omega}_{m}(\pi,G)d\tau(m).
 	\end{align*}
 
  Then the proof  can be proved by using Theorem \ref{thm 1.1}, Proposition \ref{prop 4.1} and   ergodic decomposition for $h^{\omega}_{\mu}(\pi,G).$   
 \end{proof}

\begin{df}
	Let  $\pi:(X,G)\rightarrow (Y,G)$ be a factor map between two $G$-systems with $h^\omega_{top}(\pi,G)<\infty$, $\omega \in (0,1]$ and  $f\in C(X,\mathbb{R})$.  A tangent functional to $P^{\omega}(\pi,\cdotp,G)$ at $f$ is a finite signed measure $\mu:\mathcal{B}(X)\rightarrow \mathbb{R}$ such that $$P^{\omega}(\pi,f+h,G)-P^{\omega}(\pi,f,G)\geq \omega\int h d\mu$$ for all $ h\in C(X,\mathbb{R})$.
	 Let $t_f(X,G)$ denote the collection of all tangent functionals to $P^{w}(\pi,\cdotp,G)$ at $f$.
\end{df}

\begin{rem}
	\begin{enumerate}
		\item  Riesz representation states that  each $L\in C(X,\mathbb{R})^*$  is  the form of $L(f)=\int f d\mu$ for any $f\in C(X,\mathbb{R})$.  Then  $C(X,\mathbb{R})^*$ can be identified with the set of all finite signed measures on $(X,\mathcal{B}(X))$. Hence the tangent functionals to $P^{w}(\pi,\cdotp,G)$ at $f$ can be regarded as  the elements of $C(X,\mathbb{R})^*$ satisfying $\omega L(h)\leq P^{\omega}(\pi,f+h,G)-P^{\omega}(\pi,f,G)$ for all $h\in C(X,\mathbb{R})$.
		\item   $t_f(X,G)$ is not empty. Let the constant  linear function space $C:=\{c:c\in \mathbb{R}\}$ on $X$. Then  we have $\gamma(c):=\omega\int f d\mu\leq P^{\omega}(\pi,f+c,G)-P^{\omega}(\pi,f,G)$ for all $c\in C$. Using Hahn-Banach theorem,  one can extend $\gamma$ on $C$ to an element of $C(X,\mathbb{R})^{*}$ dominated by the convex function $h \mapsto P^{\omega}(\pi,f+h,G)-P^{\omega}(\pi,f,G)$.
	\end{enumerate}
\end{rem}

\begin{thm}\label{th5.3}
Let  $\pi:(X,G)\rightarrow (Y,G)$ be a factor map between two $G$-systems with $h^\omega_{top}(\pi,G)<\infty$, $\omega \in (0,1]$ and  $f\in C(X,\mathbb{R})$.  Then
\begin{enumerate}
	\item   ${M}_{f}(X,G)\subset t_f(X,G)\subset {M}(X,G)$.
	\item    If the entropy map $\mu\in M(X,G)\mapsto h^{\omega}_{\mu}(\pi,G)$ is upper semi-continuous on  $t_f(X,G)$,  then ${M}_{f}(X,G)= t_f(X,G)$. 
	\item If the entropy map $\mu\in M(X,G)\mapsto h^{\omega}_{\mu}(\pi,G)$ is upper semi-continuous on  $M(X,G)$,  then  there is a dense subset  $\mathcal{S}$ of $C(X,\mathbb{R})$ such that ${M}_{f}(X,G)$ is a singleton  for  each $f\in \mathcal{S}$. 
\end{enumerate}
 
\end{thm} 
\begin{proof}
	(1) Let $\mu \in {M}_{f}(X,G)$. If $h\in C(X,\mathbb{R})$, then by Theorem \ref{thm 1.1}
	\begin{align*}
		P^{\omega}(\pi,f+h,G)-P^{\omega}(\pi,f,G)&\geq h^\omega_{\mu}(\pi,G)+\omega\int f +hd\mu\\
		&-h^\omega_{\mu}(\pi,G)-\omega\int f d\mu\\
		&=\omega\int h d\mu.
	\end{align*}
So  ${M}_{f}(X,G)\subset t_f(X,G)$.	
	
	We show $t_f(X,G)\subset {M}(X,G)$. Let $\mu \in t_{f}(X,G)$. Let  $h\geq 0$ and $\epsilon >0$. We have 
	\begin{align*}
		\int (h+\varepsilon)d\mu &=-\int -(h+\varepsilon)d\mu\\
		&\geq -P^{\omega}(\pi,f-(h+\varepsilon))+P^{\omega}(\pi,T,f)\\
		&\geq-\left\lbrace P^{\omega}(\pi,T,f)-\omega\inf (h+\varepsilon)\right\rbrace +P^{\omega}(\pi,T,f)\\
		&=\omega\inf (h+\varepsilon)>0.
	\end{align*}
	Then $\int h d\mu \geq0$  and hence  $\mu$ takes non-negative values. 
	
	We  show $\mu (X)=1$. If $n\in \mathbb{Z}$, then $\omega\int n d\mu\leq P^{\omega}(\pi,f+n,G)-P^{\omega}(\pi,f,G)=\omega n$.  This shows that if $n\geq1$ then $\mu (X)\leq1$ and if $n\leq-1$ then $\mu (X)\geq1$.
	
	We show $\mu \in {M}(X,G)$. If $n\in \mathbb{Z}$, $h\in C(X,\mathbb{R})$ and for all $g\in G$
	\begin{align*}
		\omega n\int (h\circ g- h) d\mu\leq P^{\omega}(\pi,f+n(h\circ g- h),G)-P^{\omega}(\pi,f,G))=0.
	\end{align*}
	If $n>0$, then  $\int h\circ g d\mu \leq \int h d\mu$. If $n<0$,  then $\int h\circ g d\mu \geq \int h d\mu$.  This shows that  $\int h\circ g d\mu = \int h d\mu$ and hence $\mu \in {M}(X,G)$. 
 
	(2) We  show ${M}_{f}(X,G)=t_f(X,G)$.  It suffices to show ${t}_{f}(X,G)\subset M_f(X,G)$. Let  $\mu \in t_{f}(X,G)$.  Then $$P^{\omega}(\pi,f+h,G)-\omega\int (f+h) d\mu \geq P^{\omega}(\pi,f,G)-\omega\int f d\mu $$ 
	 for all $ h\in C(X,\mathbb{R})$. Notice that $h$ is arbitrary.  By  Theorem \ref{th4.4}, we  have  $h^\omega_{\mu}(\pi,G)\geq P^{\omega}(\pi,f,G)-\omega\int f d\mu$  and hence  $\omega\int f d\mu +h^\omega_{\mu}(\pi,G)\geq P^{\omega}(\pi,f,G)$. Together with Theoroem \ref{thm 1.1},  $\mu \in {M}_{f}(X,G)$.
	 
	 (3) By (2),  we know that ${M}_{f}(X,G)= t_f(X,G)$. Then (3) holds by the  fact that \cite{ds58} the convex function $P^{\omega}(\pi,\cdot,G)$ on a separable space Banach space $C(X,\mathbb{R})$ has a unique  tangent functional at a dense set of points.
\end{proof}

\section*{Appendix} \label{thma}
Recall that a $\mu \in M(X)$ is said   to be a   \emph{$G$-invariant Borel probability measure}  if $\mu(A)=\mu(gA)$ for  all $A\in \mathcal{B}_X$ and $g\in G$. A measure  $\mu \in M(X,G)$ is said to be \emph{$G$-ergodic} if $G$-invariant Borel set   B (i.e $gB=B$ for all $g \in G$)  has measure $0$ or $1$. It is well-known that $E(X,G)$ is exactly the extreme points of $M(X,G)$.  Although the following  theorem is well-known for $\mathbb{Z}$-actions,  the proof seems to  to be absent for amenable group actions. Here, we provide a  proof for  completion.
\begin{thm1}
	Let $\pi:(X,G)\rightarrow (Y,G)$ be a factor map between two $G$-systems.  Then
	$$\pi_{*}(M(X,G))=M(Y,G),
	\pi_{*}(E(X,G))=E(Y,G)$$
\end{thm1}

\begin{proof}
	It is clear  that $\pi_{*}(M(X,G))\subset M(Y,G)$ by definition.  Let $\nu\in M(Y,G)$. We define a linear functional  $L_{\nu}$ on the subspace $\{f\circ\pi:f\in C(Y,\mathbb{R})\}$ of $C(X,\mathbb{R})$  by
	$$L_{\nu}(f\circ\pi )=\int fd\nu.$$ Then  $L_{\nu}(1)=1$ and  $||L_{\nu}||=1$. By Hahn-Banach theorem,  there  exists a  continuous linear functional $L$ on $C(X,\mathbb{R})$ so that $||L||=1$ and  $L(f\circ\pi)=L_{\nu}(f\circ\pi).$  Let $f\in C(X,\mathbb{R})$ with $f\geq 0$ and put $a:=||f||_{\infty}+1$. Notice that  $f\geq 0$. Then  $|L(a-f)|\leq ||a-f||_{\infty}\leq a$ and hence
	\begin{align*}
		L(f)=L(f-a+a)\geq a-|L(a-f)|\geq0.
	\end{align*}
	
	By  Riesz's representation theorem, there exists $\mu_0\in M(X)$ such that $L(f)=\int fd\mu$ for all $f\in C(X,\mathbb{R})$.  This yields that for  each $f\in C(Y,\mathbb{R})$,
	$$\int fd\pi_{*}\mu_0=\int f\circ\pi d\mu_0=L(f\circ\pi)=L_{\nu}(f\circ\pi)=\int fd\nu.$$
	Hence, $\pi_{*}\mu_0=\nu$. 
	
	Let $\{F_n\}_{n\ge1}$ be a F\o lner sequence of $G$  and define
	$$\mu_n=\frac{1}{|F_n|}\sum_{g\in F_n}\mu_0\circ g^{-1}.$$
	Without loss of generality, we assume that $\mu_n \mapsto \mu\in M(X,G)$ in weak$^{*}$ topology. Then for  each $f\in C(Y,\mathbb{R})$,
	
	\begin{align*}
		\int fd\pi_{*}\mu=\int f\circ\pi d\mu&=\lim\limits_{n \to \infty}\int f\circ\pi d\mu_n= \lim\limits_{n \to \infty} \frac{1}{|F_n|}\sum_{g\in F_n}\int f\circ\pi\circ g d\mu_0\\
		&=\lim\limits_{n \to \infty} \frac{1}{|F_n|}\sum_{g\in F_n}\int f\circ g\circ\pi d\mu_0=\int fd\nu.
	\end{align*}
	This shows that $ M(Y,G)\subset \pi_{*}(M(X,G))$.  
	
	The left for us is to show $E(Y,G)\subset \pi_{*}(E(X,G))$. Let $\nu\in M(Y,G)$. Then $\pi^{-1}(\nu)$ is a non-empty  closed  convex set. Moreover,   the extreme point of  $\pi^{-1}(\nu)$ is  also  the extreme point of $M(X,T)$. Otherwise,  there is a extreme point $\mu \in \pi^{-1}(\nu)$ such that  $\mu=p\mu_1+(1-p)\mu_2$, where   $0<p<1$ and $\mu_1,\mu_2\in M(X,T)$ do not belong to  $\pi^{-1}(\nu)$. Then
	$$\nu=\pi_{*}\mu=p\pi_{*}\mu_1+(1-p)\pi_{*}\mu_2,$$
	a contradiction.
	
\end{proof}

\section*{Acknowledgement} 

\noindent The work was supported by the
National Natural Science Foundation of China (Nos.12071222 and 11971236). The third author was also supported by Postgraduate Research $\&$ Practice Innovation Program of Jiangsu Province (No. KYCX23$\_$1665).  The work was also funded by the Priority Academic Program Development of Jiangsu Higher Education Institutions.  We would like to express our gratitude to Tianyuan Mathematical Center in Southwest China(No.11826102), Sichuan University and Southwest Jiaotong University for their support and hospitality.  


\end{document}